\documentclass[a4paper,reqno]{amsart}
\usepackage{amssymb}
\usepackage{amsmath}
\usepackage{amsthm}
\usepackage{pgfplots}
\usepackage{hyperref}
\hypersetup{
    colorlinks=true,
    linkcolor=black, 
    urlcolor=black, 
    citecolor=black 
}
\usepackage{tikz}
\usetikzlibrary{calc}
\usetikzlibrary{angles,quotes}
\usetikzlibrary{decorations.pathmorphing}

 \allowdisplaybreaks \numberwithin{equation}{section}
\begingroup
\theoremstyle{plain}
\newtheorem{theorem}{Theorem}[section]
\newtheorem{proposition}[theorem]{Proposition}
\newtheorem{lemma}[theorem]{Lemma}
\newtheorem{corollary}[theorem]{Corollary}

\theoremstyle{definition}

\newtheorem{remark}[theorem]{Remark}

\endgroup

\def \dist {\mathop {\rm dist}\nolimits}
\def \e {\epsilon}

\title[{Three quantitative versions of the P\'al inequality}]{Three quantitative versions of the P\'al inequality\!}
\author{Ilaria Lucardesi}
\address{Dipartimento di Matematica, Università di Pisa, Largo Bruno Pontecorvo 5, 56127 Pisa, Italy}
\email{ilaria.lucardesi@unipi.it}

\author{Davide Zucco}
\address{Dipartimento di Matematica ``G. Peano'', Università di Torino, Via Carlo Alberto 10, 10123 Torino, Italy}
\email{davide.zucco@unito.it}

\begin{document}

\begin{abstract}
The P\'al inequality is a classical result which asserts that among all planar convex sets of given width the equilateral triangle is the one of minimal area. In this paper we prove three \emph{quantitative} versions of this inequality, by quantifying  
how the closeness of the area of a convex set, of certain width, to the minimal value implies its closeness to the equilateral triangle. 
As a by-product, we also present a novel result concerning a quantitative inequality for the inradius of a set, under minimal width constraint.
\end{abstract}

\maketitle


\section{Introduction}

The Kakeya needle problem, initially proposed by Kakeya, addresses the intriguing question of whether there exists a set in the plane, of minimal area, within which a needle of length 1 can freely rotate. Kakeya posed this question in 1917 within the restricted class of convex sets. This specific inquiry was solved in 1921 by P\'al, identifying the equilateral triangle of height 1 as the unique convex set capable of accommodating such rotations. On the other hand, Besicovitch extended the investigation to a broader class of sets, removing the convexity assumption, demonstrating that there is no lower bound for the required area: sets of arbitrarily small area can indeed allow a full $360$° rotation of the needle. Over the past decades, a growing realization has dawned upon the mathematical community: this kind of Kakeya needle problems is intricately intertwined with numerous seemingly disparate fields, including Harmonic Analysis, Number Theory, Geometric and Arithmetic Combinatorics, see \cite{tao} and references therein. 
The free rotations of a needle of given length inside a set can be formulated by imposing a constraint on the \emph{minimal width} of the set. Indeed P\'al's solution to Kakeya needle problem is obtained via an inequality bounding from below the area of convex sets of given width. More precisely, for every planar convex body (i.e., compact set with non-empty interior) $K$ there holds
$$
\frac{|K|}{\omega^2(K)} \geq \frac{1}{\sqrt{3}},
$$
where $|K|$  stands for the area ($2$-dimensional Lebesgue measure) of $K$ and $\omega(K)$ its minimal width (minimal distance between two parallel supporting lines). The minimal width of a convex set holds significant importance also in Geometric Tomography, particularly when seeking information about sections or projections of sets, see \cite{gar}.
In the previous formulation, the minimal width constraint is incorporated by dividing by $\omega^2$,  resulting in a scale-invariant left-hand side (it does not change when the set $K$ is scaled). A few proofs of this inequality can be found in the literature: the original geometric proof of P\'al in \cite{pal},  revised in \cite{yagbol}, another one in \cite{cacogr} and a more analytic proof in \cite{mal}. As far as we are aware, these are the only known proofs.
This \emph{P\'al inequality} is sometimes also referred to as \emph{isominwidth inequality} \cite{came} due to its striking resemblance to a more ancient and studied inequality: the \emph{isoperimetric inequality} (see for instance \cite{fus} and references therein). It's worth noting that replacing the perimeter constraint with the minimal width introduces several significant differences. First, the width constraint shifts the focus towards the minimization of the area instead of its maximization. More importantly, the introduction of the minimal width results in a \emph{breaking of the full symmetry}: the minimizer is no longer a disk (it is in fact an equilateral triangle). The interaction of the minimal width with some spectral functionals will be the subject of a forthcoming paper (see \cite{bog} for some conjectures in this direction).

In this paper, our focus is primarily on \emph{quantitative versions of the P\'al inequality}: 
inequalities for which
closeness of the area of a set of given width to the minimal value $1/\sqrt{3}$ implies closeness of the set to the equilateral triangle.
More precisely, we would bound from below the \emph{deficit}
$|K|/ \omega^2(K) - 1/\sqrt{3}$
and 
quantify, with some kind of distance, how close $K$ is to an equilateral triangle. 
This quantification will be achieved through the examination of different types of distances. 

The first quantitative version of the P\'al inequality which we derive is a result in the spirit of Bonnesen \cite{bon}, where we bound the deficit with a quantity depending on the inradius $r(K)$ of a set $K$ (these estimates were named \emph{Bonnesen type inequalities} by Osserman in \cite{oss}).
It is well-known that for every planar convex body $K$ there holds ${r(K)}/{\omega(K)}\geq 1/3$ and equality holds only when $K$ is an equilateral triangle. Therefore, with the quantity
\begin{equation}\label{defeta}
\eta(K):=\frac{r(K)}{\omega(K)}-\frac{1}{3}
\end{equation}
we can measure how $K$ is far from being an equilateral triangle.

\begin{theorem}[Quantitative P\'al inequality via inradius]\label{teo1}
There exists a constant $c_1>0$, such that for every planar convex body $K$ there holds
\begin{equation}\label{eq1}
\frac{|K|}{\omega^2(K)} - \frac{1}{\sqrt{3}} \geq c_1\eta(K), 
\end{equation}
with $\eta$ as in \eqref{defeta}.
\end{theorem}

The second version of quantitative P\'al inequality is based on the Hausdorff distance $d_\mathcal H$ of sets: given a planar convex body $K$ let
\begin{equation}\label{alphaT}
\alpha_E(K):=\min\left\{\frac{d_{\mathcal H}(K,E)}{\omega(E)}\ :\ E\hbox{ equilateral triangle},\ \omega(E)=\omega(K)\right\}.
\end{equation}
This is a measure of the smallest enlargement, either of $K$ or of a ``best'' equilateral triangle $E$ for $K$, in order to include the other set.

\begin{theorem}[Quantitative P\'al inequality via Hausdorff distance]\label{teo2}
There exists a constant $c_2>0$ such that, for every planar convex body $K$ there holds
\begin{equation}\label{statement1}
\frac{|K|}{\omega^2(K)} - \frac{1}{\sqrt{3}} \geq c_2\, \alpha_E(K),
\end{equation}
with $\alpha_E$ as in \eqref{alphaT}.
\end{theorem}

The third quantitative inequality is given in terms of what is usually called \emph{Fraenkel asymmetry}, which we define for every planar convex body $K$ as the quantity
\begin{equation}\label{AT}
\mathcal A_E(K):=\min\left\{\frac{|K \Delta E|}{\omega^2(E)}\ :\ E \hbox{ equilateral triangle},\ \omega(E)=\omega(K)\right\},
\end{equation}
where $K \Delta E$ denotes the symmetric difference of the sets $K$ and $E$, namely $K \Delta E=(K\setminus E)\cup (E\setminus K)$. This asymmetry measures the area of the symmetric difference between $K$ and a ``best'' equilateral triangle.

\begin{theorem}[Quantitative P\'al inequality via Fraenkel asymmetry]\label{teo3}
There exists a constant $c_3>0$ such that, for every planar convex body $K$ 
there holds
\begin{equation}\label{statement2}
\frac{|K|}{\omega^2(K)} - \frac{1}{\sqrt{3}} \geq {c_3}\, \mathcal A_E(K),
\end{equation}
with $\mathcal A_E$ as in \eqref{AT}.
\end{theorem}

The definition we introduced for $\alpha_E$ and $\mathcal A_E$ are inspired to the classical {\it asymmetries}  
$$
\alpha_D(K):=\min\left\{\frac{d_{\mathcal H} (K,D)}{\sqrt{|D|}}\ :\ D \hbox{ disk},\ |D|=|K| \right\}
$$
and
$$
\mathcal A_D(K):=\min\left\{\frac{|K\Delta D|}{|D|}\ :\ D \hbox{ disk},\ |D|=|K| \right\},
$$
appearing in the quantitative isoperimetric inequalities (see, e.g., \cite{bicrhe, cicleo, fug, fus, fumapr}), in which the equilateral triangle is replaced by a disk and the minimal width constraint by the area constraint. These two asymmetries quantify, is some sense, the distance of a set from being circular. The asymmetry \eqref{AT} has been already used in \cite{ind} for a quantitative Faber--Krahn inequality for triangles.

In quantitative inequalities two main topics are of particular interest: the sharpness of the exponent and the computability of the constant appearing, respectively, in the power and in front of the distance under consideration (the sharp value of the constant is always a challenging problem that can be solved in very few cases, see \cite{cicleo} and references therein).

\begin{remark}[Sharpness of the exponents and computability of the constants]
The exponent $1$ on $\eta$, $\alpha_E$, and $\mathcal A_E$ in the right-hand sides of \eqref{eq1}, \eqref{statement1}, and \eqref{statement2} is sharp, in the sense that these quantitative inequalities are no longer true if the exponent is lowered. Moreover, the constants (though not optimal) are explicit and can be taken equal to the following ones: $
c_1= {1}/{\sqrt{5}}$, $c_2={1}/{(25\sqrt{5})}$, and $c_3=  {1}/{(25(3\sqrt{3} + 2)\sqrt{5})}$.
\end{remark}

The paper is organized as follows. We start in Section~\ref{sec-prel} with some definitions and preliminary facts about the minimal width, the inradius, and the Hausdorff distance. In Section~\ref{sec-fromator} we relate the minimization of the area to the one of the inradius, under minimal width constraint, then we prove Theorem~\ref{teo1}. 
In Section~\ref{sec-r} we derive a new result about a quantitative inequality for the inradius, under minimal width constraint. 
We use this new result to prove, in Section~\ref{sec-a}, Theorem~\ref{teo2} and Theorem~\ref{teo3}. In the last Section~\ref{sec-sharp} we prove the sharpness of the exponents of all the quantitative inequalities discussed in the paper. 

\section{Estimates of the Hausdorff asymmetry}\label{sec-prel}

In this section, we begin by establishing the notation and by proving some preliminary results for the asymmetries required in the quantitative estimates.

We work with planar \emph{convex bodies}, namely with planar compact convex sets with non-empty interior. The shape functionals that we consider are the \emph{area} $|\cdot |$ (i.e., the Lebesgue measure) and the \emph{minimal width} $\omega(\cdot)$. The latter is defined, for a planar convex body, as the minimal distance between two parallel supporting lines enclosing the set. Equivalently, the minimal width is the minimal orthogonal projection of the convex body onto a line: given a planar convex body $K$ then
\[
\omega(K)=\min_{0\leq \theta\leq \pi} \mathrm{proj_\theta}(K),
\]
where $\mathrm{proj}_\theta$ denotes the length of the orthogonal projection of $K$ onto the line $r_\theta = \{y = (\tan \theta) x\}$. Since we will only work with the width with respect to directions for which it is minimal (and not with respect to other directions), we will simply refer to this quantity as the \emph{width}.

An auxiliary shape functional that we also need is the \emph{inradius} $r(\cdot)$, that is the radius of the largest disk (that we call \emph{indisk}) contained into the set. Since $K$ is a convex body it is clear that this disk always exists. 
We recall a fundamental result regarding the boundary points of a convex body that also belong to the inscribed disk. These points are referred to as \emph{contact points}. For a proof see, e.g., \cite[Exercise 6-2]{yagbol} (this just relies on suitable translations and enlargements of disks contained in $K$). 
\begin{lemma}[Contact points]\label{lemma-contact}
Let $K$ be a planar convex body and let $D\subset K$ be a disk of radius $r(K)$. Then the boundaries $\partial K$ and $\partial D$ meet
\begin{itemize}
\item[(i)] either in two diametrically opposite points (which occurs, e.g., for rectangles and disks);
\item[(ii)] or in three points forming the vertices of an acute angled triangle.
\end{itemize}
\end{lemma}

To fix the notation we will need in the following, let us introduce the following definitions. Consider a planar convex set $K$ with an indisk $D$ satisfying condition (ii) in Lemma~\ref{lemma-contact}, i.e., such that the boundaries $\partial K$ and $\partial D$ meet in three points  $P_1$, $P_2$, and $P_3$ forming the vertices of an acute angled triangle.

For every $1\leq i\leq 3$ let $\ell_i$ be the line tangent to $D$ at $P_i$. Each line $\ell_i$ splits the plane into two half-planes and, in view of the convexity of $K$, we have that $K$ and $B$ lie in the same half-planes. Since $P_1P_2P_3$ is an acute angled triangle, we can infer that the intersection of the three half-planes containing $B$ and $K$ is a triangle, that we call a \emph{circumscribed triangle $T_K$ of $K$}. For $1\leq i\leq 3$ we also denote by $V_i$ the vertex of $T_K$ opposite to $P_i$, and by $\gamma_i$ the angle of $T$ at $V_i$. Finally, denoting by $\omega_i$ the directional width of $K$ in the direction orthogonal to $\ell_i$, we consider the points of $K$ lying on the supporting line parallel to $\ell_i$ at distance $\omega_i$: among these points (there esists at least one) we denote by $Q_i$ the closest to $V_i$, see Figure~\ref{fig.KDT}.

The width of a triangle is the smallest of the three heights, associated with the largest side-length. Denoting the side-lengths of $T_K$ by $s_1\geq s_2\geq s_3$ and by $h_1\leq h_2\leq h_3$ the corresponding heights. Now, by construction 
\begin{equation}\label{wrkt}
\omega(K)\leq \omega(T_{K})=h_1 \quad \text{while}\quad  r(K)=r(T_{K}).
\end{equation}
A sufficient criterion for the existence of this circumscribed triangle $T_K$ to $K$ is the \emph{strict} inequality $r(K)< \omega(K)/2$. Indeed if $K$ satisfies such an estimate then $2r(K)<\omega(K)$ and $K$ cannot have two contact points on the same diameter of the indisk $D$.
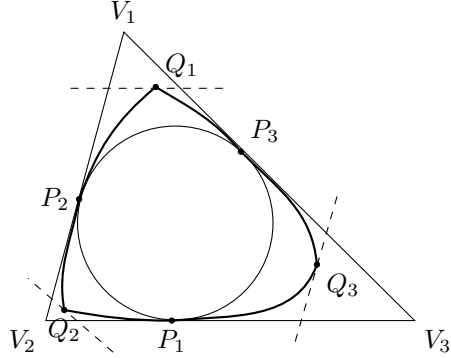
\begin{figure}[t]
        \centering
\begin{tikzpicture}[scale=1.4]
\coordinate (V1) at (120:2);
\coordinate (V2) at (210:2);
\coordinate (V3) at (-30:2);
\coordinate (P3) at (0.1,0.6);
\coordinate (P2) at (-1.42,0.15);
\coordinate (P1) at (-0.55,-1);
\coordinate (Q1) at (120:1.4);
\coordinate (Q2) at (210:1.8);
\coordinate (Q3) at (-30:0.94);
\draw (V1) -- (V2) -- (V3) -- cycle;
\node[above] at (V1) {$V_1$};
\node[below left] at (V2) {$V_2$};
\node[below right] at (V3) {$V_3$};
\node[above right] at (P3) {$P_2$};
\node[left] at (P2) {$P_3$};
\node[below] at (P1) {$P_1$};
\draw[fill] (0.1,0.6) circle (0.7pt);
\draw[fill] (-1.42,0.15) circle (0.7pt);
\draw[fill] (-0.55,-1) circle (0.7pt);
\node[above right] at (Q1) {$Q_1$};
\node[below] at (Q2) {$Q_2$};
\node[below right] at (Q3) {$Q_3$};
\draw[fill] (Q1) circle (0.7pt);
\draw[fill] (Q2) circle (0.7pt);
\draw[fill] (Q3) circle (0.7pt);
\coordinate (center) at (barycentric cs:V1=1.15,V2=1.3,V3=0.95);
\pgfmathsetmacro{\radius}{1.06 * sin(60)}
\draw[thin] (center) circle (\radius);
\draw[dashed] (-1.5,1.2) -- (0,1.2);
\draw[dashed] (0.61,-1.2) -- (1.01,0.2);
\draw[dashed] (-1.1,-1.3) -- (-1.9,-0.6);
\draw[thick] (Q1) to[out=220,in=70] (P2) to[out=260,in=100] (Q2) to[out=350, in=180] (P1) to[out=5,in=250] (Q3) to[out=95,in=315] (P3) to[out=130,in=329] (Q1);
\end{tikzpicture}
 \caption{{In bold it is shown the boundary $\partial K$ of $K$. The circle is the indisk $D$ of $K$ and the triangle $V_1V_2V_3$ is the circumscribed triangle $T_K$. The dashed lines are parallel to $\ell_i$ at distance $\omega_i$.}}\label{fig.KDT}
\end{figure}

Throughout the paper, we will occasionally use the same notation for both objects and their corresponding values, with a slight abuse of notation. For instance, the symbol $s$ may refer to both a side and the length of a triangle's side; similarly, $\omega$ may denote both the height of a triangle and its width (i.e., the length of the height). Additionally, we will use the same symbol $\omega$ to represent both the width function and the width of a set, and so forth for the other shape functionals. For brevity, we will also omit the subscript $E$ in the asymmetries \eqref{alphaT} and \eqref{AT}, simply denoting them as $\alpha$ and $\mathcal{A}$. 

A natural topology in the class of convex bodies is obtained via the \emph{Hausdorff distance} $d_{\mathcal H}$. Given a planar convex body $K$ the $\e$-tubular neighborhood of $K$ is $(K)_\e:=\{x \in \mathbb R^2 \ :\  \mathrm{dist}(x,K) \leq \e \ \}$, namely $(K)_\e = K \oplus B_\epsilon$ the Minkowski sum of $K$ and a disk $B_\epsilon$ of radius $\epsilon$ (see e.g., \cite[Chapter 3]{rolf}). Then the Hausdorff distance of two planar convex bodies $K$ and $J$ is defined as
$$
d_{\mathcal H}(K,J):=\min\{\e > 0 \ :\ K \subset (J)_\e\,,\ J\subset (K)_\e\}.
$$
Roughly speaking this distance measures the smallest enlargement of a set in order to include the other.  We will also need the $\e$-retraction of $K$ defined by
$$
(K)_{-\e}:=\{x\in K \ : \  \mathrm{dist}(x,\partial K)\geq \e  \}.
$$
This is also known in the literature as {\it inner parallel set} and can be also written as the Minkowski difference $(K)_{-\epsilon}=K \sim B_\epsilon$. 
The set $(K)_{-\epsilon}$ is not empty provided that $\epsilon \leq r(K)$. Notice that the Minkowski addition and subtraction are not one the inverse of the other, in the following sense: $(K_\epsilon)_{-\epsilon}= K$ while $(K_{-\epsilon})_\epsilon \subset K$ and the inclusion can be strict (see, e.g. \cite[pp. 146-147]{rolf}).
For an equilateral triangle $E$, we have the following lemma.
\begin{lemma}
Let $E$ be an equilateral triangle. Let $\epsilon < r(E)$. Then
\begin{equation}\label{contE}
(E)_\epsilon \subset \left(1+\frac{\e}{r(E)}\right)E, \quad (E)_{-\e}=\left(1-\frac{\e}{r(E)}\right)E.
\end{equation} 
\end{lemma}
\begin{proof}
The first statement follows using the structure of the Minkowski sum $(E)_\epsilon=E\oplus B_\epsilon$: the boundary is made of 3 segments parallel to the sides of $E$, with the same length of the sides of $E$, and whose endpoints are joined by 3 arcs of circle of radius $\epsilon$. The set is contained into an equilateral triangle, concentric to $E$, of the form $\lambda E$ for some $\lambda>1$. On the one hand, by construction, the inradius of $\lambda E$ is $r(E) + \epsilon$; on the other hand, by the homogeneity of the inradius, we have $r(\lambda E) = \lambda r(E)$. Therefore we deduce
$$
\lambda = 1 + \frac{\epsilon}{r(E)}.
$$
The proof of the second statement is similar: it is immediate to check that $(E)_{-\epsilon}$ is an equilateral triangle, concentric to $E$, of the form $\mu E$, for some $0 <\mu <1$. Then, we use $r(\mu E)=\mu r(E) = r(E) - \epsilon$ to deduce the value of the scaling factor $\mu$.
\end{proof}

We begin by providing the well-posedness of these asymmetries.

\begin{lemma}\label{lem.min}
The asymmetries $\alpha$ and $\mathcal A$ are well-defined, namely the minima in \eqref{alphaT} and \eqref{AT} are attained. 
\end{lemma}
\begin{proof}
We give the proof for $\alpha$, since that for $\mathcal A$ is similar. Let $K$ be a fixed planar convex body. 
The equilateral triangles in the plane of width $\omega(K)$ just depend on the position $x\in \mathbb R^2$ of the baricenter and on the angle $\theta \in [0,2\pi]$ of one axis of symmetry with respect to the horizontal line. Let $E_{x,\theta}$ denote such a triangle. Then the minimum problem defining $\alpha$ in \eqref{alphaT} is equivalent to the following one:
$$
\min\left\{d_{\mathcal H}(K,E_{x,\theta})\ :\ x\in \mathbb R^2\,,\ \theta \in [0,2\pi]\right\}.
$$
Since the triangles $E_{x,\theta}$ that do not intersect $K$ are certainly not optimal, we can restrict $x$ to a compact set (for example, we can take $x$ in a disk of radius proportional to the diameter of $K$, which is finite due to the boundedness of $K$). Therefore, the minimization problem is on a compact set, the function $(x,\theta)\mapsto d_{\mathcal H}(K, E_{x,\theta})$ is continuous and the existence of minima follows.
\end{proof}

Our main objective is to derive estimates on the asymmetries of both generic and specific convex bodies. To achieve this goal, we establish estimates on the Hausdorff distance of certain convex bodies.

\begin{lemma}\label{lem.triangle}
Let $T$ be a triangle of width $\omega(T)$ and of smallest interior angle $\gamma_T$. Let $E$ be an equilateral triangle of width $\omega(E)\leq \omega(T)$ aligned to the width $\omega(T)$ of $T$ and with the corresponding base on the largest segment of $T$. Then 
\[
d_\mathcal H (T, E)\leq \frac{\sqrt{3}}{\tan \gamma_T}\omega(T)-\omega(E).
\]
\end{lemma}
\begin{proof}
We denote by $h_1\leq h_2\leq h_3$ the heights of $T$ of the corresponding sides $s_1\geq s_2\geq s_3$. For $1\leq j\leq 3$ let $V_j$ be the vertex of $T$ opposite to $s_j$ and $\gamma_j$ be the angle of $T$ at $V_j$ (we use for $T$ the same notation introduced before for $T_K$). In this simplified notation $h_1=\omega(T)$ and $\gamma_3=\gamma_T$. Moreover, we denote $\omega=\omega(E)$.  
Let $W_1$ be the vertex of $E$ on $h_1$ and $W_2, W_3$ the remaining ones so that
$$
d_{\mathcal{H}}(T,E) \leq \max_{1\leq j\leq 3}\mathrm{dist}(V_j,W_j).
$$
Let $O$ be the projection of $V_1$ on $s_1$ (an endpoint of $h_1$). See Figure~\ref{fig.triangles} for a picture of this construction. 

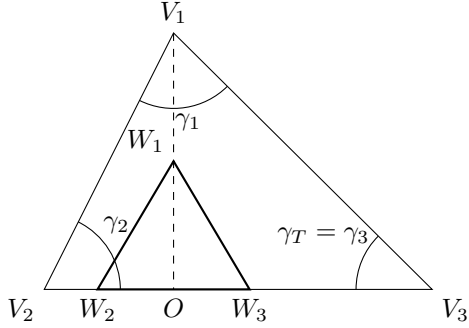
\begin{figure}[t]
        \centering
        \begin{tikzpicture}[scale=1.7]
\coordinate (W1) at (0,1);
\coordinate (W2) at (-1/1.7,0);
\coordinate (W3) at (1/1.7,0);
\coordinate (V1) at (0,2);
\coordinate (V2) at (-1,0);
\coordinate (V3) at (2,0);
\coordinate (O) at (0,0);
\draw[thick] (W1) -- (W2) -- (W3) -- cycle;
\draw (V1) -- (V2) -- (V3) -- cycle;
\draw[dashed] (V1) -- (O);
\node[below] at (O) {$O$};
\node[above left] at (W1) {$W_1$};
\node[below] at (W2) {$W_2$};
\node[below] at (W3) {$W_3$};
\node[above] at (V1) {$V_1$};
\node[below left] at (V2) {$V_2$};
\node[below right] at (V3) {$V_3$};
\draw pic[draw=black, angle eccentricity=1.2, angle radius=10mm]  {angle=V3--V2--V1};
\node at ($(W2)!0.25!(V1)$) {$\gamma_2$};
\draw pic[draw=black, angle eccentricity=1.2, angle radius=10mm] {angle=V1--V3--V2};
\node at ($(W1)!0.58!(V3)$) {$\gamma_T=\gamma_3$};
\draw pic["$\gamma_1$", draw=black, angle eccentricity=1.2, angle radius=10mm] {angle=V2--V1--V3};
\end{tikzpicture}
\caption{The triangles $T$ and $E$ of Lemma~\ref{lem.triangle}.}\label{fig.triangles}
\end{figure}

Using the fact that $V_1OV_2$ and $V_1OV_3$ are right triangles, we may infer that
$$
 \mathrm{dist}(V_1,W_1) = h_1-\omega
$$
and
$$
 \mathrm{dist}(V_2,W_2) = | s_3 \cos \gamma_2 - \omega/\sqrt{3}  |, \quad  \mathrm{dist}(V_3,W_3) = | s_2 \cos \gamma_3 - \omega/\sqrt{3}  |.
$$
We claim that 
\begin{equation}\label{dueclaim}
\mathrm{dist}(V_1,W_1)  \leq  \sqrt{3}\,\mathrm{dist}(V_3,W_3)\quad \hbox{and}\quad \mathrm{dist}(V_2,W_2)  \leq  \mathrm{dist}(V_3,W_3).
\end{equation}
Indeed, since by construction $\gamma_3$ is the smallest interior angle of $T$ then $\gamma_3\leq \pi/3$ and we have
\begin{equation*}
h_1= s_2 \sin \gamma_3 \leq s_2 \sin (\pi/3) = s_2 \frac{\sqrt{3}}{2}, \quad \cos \gamma_3 \geq \cos(\pi/3) = \frac12. 
\end{equation*}
These two estimates imply that 
$$
s_2 \cos \gamma_3 - \omega/\sqrt{3} \geq s_2/2 - \omega/\sqrt{3} \geq h_1/\sqrt{3} - \omega/\sqrt{3} \geq 0,
$$
and in particular
$$
 \mathrm{dist}(V_3,W_3) =  s_2 \cos \gamma_3 - \omega/\sqrt{3}\geq \mathrm{dist}(V_1,W_1)/\sqrt{3},
$$
that gives the first inequality claimed in \eqref{dueclaim}.
Now, by using $s_3\leq s_2$ and $\cos\gamma_2 \leq \cos \gamma_3$ we obtain
\begin{equation}\label{prooft1}
s_3 \cos \gamma_2 - \omega/\sqrt{3} \leq s_2 \cos \gamma_3 - \omega/\sqrt{3} .
\end{equation}
Moreover, using the fact that the largest angle $\gamma_1$ is at least $\pi/3$, we get
$$
 s_2 \cos \gamma_3 + s_3 \cos \gamma_2  = s_1 \geq \frac{2h_1}{\sqrt{3}} \geq \frac{2\omega}{\sqrt{3}},
$$
which implies
\begin{equation}\label{prooft2}
- (s_3 \cos \gamma_2 - \omega/\sqrt{3} ) \leq   s_2 \cos \gamma_3 - \omega/\sqrt{3}.
\end{equation}
By combining \eqref{prooft1} with \eqref{prooft2}, we also obtain the second inequality claimed in \eqref{dueclaim}. 
Therefore,
\[
d_{\mathcal{H}}(T,E) \leq \sqrt{3}\, \mathrm{dist}(V_3,W_3) 
\]
and the lemma is proved after noticing that $s_2\cos \gamma_3=h_1/\tan\gamma_3$ and recalling that $h_1=\omega(T)$.
\end{proof}

\begin{lemma}\label{lem.convex}
Let $K$ be a planar convex body such that $r(K)< \omega(K)/2$. Then let $T$ be the triangle circumscribed to $K$ of sides $s_1\geq s_2\geq s_3$ of width $\omega(T)$. Then 
\begin{equation}\label{eq.proof1}
d_{\mathcal H}(K,T) \leq  (s_1 - s_3) + s_3\left( \frac{\omega(T)-\omega(K)}{\omega(T)}\right).
\end{equation}
\end{lemma}
\begin{proof}
By convexity of $K$ and since $K\subset T$, the Hausdorff distance between $K$ and $T$ satisfies
$$
d_{\mathcal H}(K,T) \leq \max_{1\leq i\leq 3} \mathrm{dist}(V_i, Q_i),
$$
where $V_i$ are the vertexes of $T$ and $Q_i$ are the boundary points of $K$ introduced at the beginning of this section, see Figure~\ref{fig.KDT}.
The case that gives greater values in the previous maximum is when the point $Q_i$ lies on the sides of $T$ and when the directional width $\omega_i$ of $K$ is $\omega$, for every $1\leq i \leq 3$.
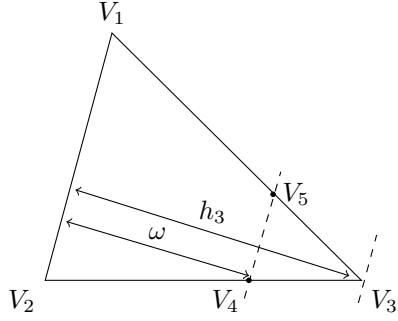
\begin{figure}[t]
        \centering
        \begin{tikzpicture}[scale=1.2]
\coordinate (V1) at (120:2);
\coordinate (V2) at (210:2);
\coordinate (V3) at (-30:2);
\coordinate (V4) at (0.5,-1);
\coordinate (V5) at (0.765,-0.05);
\draw (V1) -- (V2) -- (V3) -- cycle;
\node[above] at (V1) {$V_1$};
\node[below left] at (V2) {$V_2$};
\node[below right] at (V3) {$V_3$};
\node[below left] at (V4) {$V_4$};
\node[right] at (V5) {$V_5$};
\draw[fill] (V4) circle (0.7pt);
\draw[fill] (V5) circle (0.7pt);
\draw[dashed] (0.45,-1.2) -- (0.85,0.2);
\draw[dashed] (1.7,-1.25) -- (1.9,-0.5);
    \draw[<->] (-1.4,0) -- (1.6,-0.95) node[midway, above] {$h_3$};
    \draw[<->] (-1.5,-0.35) -- (0.5,-0.95) node[midway, above] {$\omega$};
\end{tikzpicture}
  \caption{The two similar triangles $V_1V_2V_3$ and $V_5V_4V_3$.}\label{fig.w} 
\end{figure}

Let us estimate the distance between $V_3$ and $Q_3$. To this aim, we consider the line parallel to the side $s_3$ at distance $\omega$: it cuts the sides $s_1$ and $s_2$ in two points, say $V_4$ and $V_5$. The two triangles $V_1V_2V_3$ and $V_5V_4V_3$ are similar (since they have the same angles), see Figure \ref{fig.w}. Therefore
${V_5V_3}/{(h_3-\omega)}={s_2}/{h_3}$ which is equivalent to 
\[
V_5V_3=\frac{s_2(h_3-\omega)}{h_3}.
\]
Combining this to the analogous equality $V_4V_3={s_1(h_3-\omega)}/{h_3}$ we have 
$$
\dist(V_3, Q_3) \leq \max\left\{ s_2\frac{h_3 - \omega}{h_3},   s_1\frac{h_3 - \omega }{h_3} \right\}.
$$ 
Now, by arguing in the same way for all the vertices, we obtain 
\begin{equation}\label{quasi}
d_{\mathcal H}(K,T) \leq \max_{1\leq i\neq j\leq 3}\left\{s_i\frac{h_j - \omega}{h_j}\right\}.
\end{equation}
By using $s_1\geq s_2\geq s_3$ and $h_3\geq h_2\geq h_1$, with the monotonicity of $x\mapsto (x-\omega)/x$, we infer that
$$
\max_{1\leq i\neq j\leq 3}\left\{s_i\frac{h_j - \omega}{h_j}\right\} = s_1\frac{h_3 - \omega}{h_3} .
$$
Inserting this equality in \eqref{quasi} and by using $s_1 h_1=s_3 h_3$, yields \eqref{eq.proof1}.
\end{proof}

\begin{proposition}[Asymmetry of a convex body]\label{asy.convex}
Let $K$ be a planar convex body such that $r(K)< \omega(K)/2$. Let $T$ be the triangle circumscribed to $K$ of sides $s_1\geq s_2\geq s_3$ and of smaller interior angle $\gamma_T$. Then we have
\[
\alpha(K)\leq  \frac{(s_1 - s_3)}{\omega(K)} + \frac{s_3(\omega(T)-\omega(K))}{\omega(K)\omega(T)}+\frac{\sqrt{3}}{\tan \gamma_T}\frac{\omega(T)}{\omega(K)}-1.
\]
\end{proposition}
\begin{proof}
We test $\alpha$ in \eqref{alphaT} with an equilateral triangle $E$ of height $\omega=\omega(K)$ aligned to $h_1$ and corresponding base on $s_1$. By applying the triangular inequality for the Hausdorff distance we obtain
$$
\alpha(K) \leq\frac{\left[ d_{\mathcal H}(K,T) + d_{\mathcal H}(T,E)\right]}{\omega},
$$
which together with Lemma~\ref{lem.convex}  and Lemma~\ref{lem.triangle} yields the estimate in the lemma (notice that $\omega(E)=\omega(K)$).
\end{proof}

\begin{remark}
The estimate in the previous proposition quantifies the asymmetry of a generic convex set in terms of three quantities: the first one measures how closely $K$ resembles to an equilateral triangle, while the second one measures the proximity of the width of $K$ to that of the circumscribed triangle, the third one is a mix of the two and measures both.
\end{remark}

\begin{lemma}\label{TBdist}
Let $D$ and $E$ be, respectively, a disk and an equilateral triangle centered at the same point.  

If $r(D)\leq \omega(E)/2$ then
\begin{equation}\label{small}
\frac{d_{\mathcal H}(D,E)}{\omega(E)} = \frac23 - \frac{r(D)}{\omega(E)}.
\end{equation}
If $r(D)\geq \omega(E)/2$ then 
\begin{equation}\label{big}
\frac{d_{\mathcal H}(D,E)}{\omega(E)} = \frac{r(D)}{\omega(E)} - \frac13.
\end{equation}
In particular, if $r(D)=\omega(E)/2$ then 
\begin{equation}\label{equal}
\frac{d_{\mathcal H}(D,E)}{\omega(E)} = \frac16.
\end{equation}

\end{lemma}
\begin{proof}
By using that $2r(D)=\omega(D)$ we have
$$
\epsilon_1:=\min\{\epsilon>0 \ :\ D\subset (E)_\epsilon\} = \omega(D) - r(D) - r(E) = r(D)-\frac{\omega(E)}{3}. 
$$
Similarly, by using that $r(E)=\omega(E)/3$,  we have
$$
\epsilon_2:=\min\{\epsilon>0 \ :\ E\subset (D)_\epsilon\} = \omega(E) - r(E) - r(D)= \frac23\omega(E)-r(D).
$$
Then, by definition of the Hausdorff distance, we obtain
$$
d_{\mathcal H}(D,E)= \max\left\{\epsilon_1, \epsilon_2\right\}.
$$
It is immediate to check that $\epsilon_1\leq \epsilon_2$ is equivalent to 
$r(D) \leq w(E)/2$, namely 
$$
\displaystyle{\max\left\{\epsilon_1, \epsilon_2\right\} =\left\{
\begin{array}{lll}
{\epsilon_2} =2\omega(E)/3 -  r(D) \quad &\hbox{if }r(D) \leq \omega(E)/2
\\
{\epsilon_1} = r(D) -\omega(E)/3 \quad &\hbox{if }r(D) \geq \omega(E)/2
\end{array}
\right.
}
$$
and both give $\omega(E)/6$ in the common case $r(D)=\omega(E)/2$.
\end{proof}

Notice that in the previous lemma we do not assumed $D$ and $E$ of same width.
We use this result in the following to compute the asymmetry of disks.

\begin{proposition}[Asymmetry of disks]\label{unsesto} For every disk $D$ we have
$$\alpha(D)=\frac16.$$
\end{proposition}
\begin{proof}
By translation and scale invariances of the asymmetry $\alpha(\cdot)$, we may assume $D$ centered at the origin of width $\omega(D)=1$. Then we consider an equilateral triangle $E_1$ centered at the origin of same width $\omega(E_1)=1$. Since $r(D)=1/2$ and $\omega(E_1)=1$, we can apply \eqref{equal} in Lemma~\ref{TBdist} to $D$ and $E_1$ to deduce $d_{\mathcal H}(D,E_1) = 1/6$ and in particular that $$\alpha(D)\leq \frac16.$$ 
By Lemma~\ref{lem.min}, there exists an equilateral triangle $E_2$ of width 1 such that
$$
d_{\mathcal H}(D,E_2)=\alpha(D) \leq \frac16.
$$
Let $\alpha=\alpha(D)$, thus by definition of the Hausdorff distance, $E_2$ is contained in the disk $ (B)_\alpha$ centered at the origin of radius $1/2 + \alpha$. Since $w(E_2)=1$ the smallest disk that contains $E_2$ is a circumscribed disk, with radius $2/3$.  
Therefore, the radius $1/2+\alpha$ must necessarily be greater than $2/3$ and this gives the reverse inequality
\[
\alpha(D)\geq 1/6,
\]
concluding the proof of the proposition.
\end{proof}

We conclude the section by estimating the Fraenkel asymmetry in terms of the Hausdorff asymmetry. This will allow us to obtain a quantitative version of the P\'al inequality in terms of the Fraenkel asymmetry.

\begin{proposition}\label{equivalence1}
For every planar convex body $K$, there holds
$$
\mathcal A(K) \leq   \left(3\sqrt{3} +2 \right) \alpha(K).
$$
\end{proposition}
\begin{proof}
By scale invariance of $\alpha$, we may assume $K$ of width $\omega(K)=1$ so that $K\subset S$ for a suitable strip $S$ of width $\omega(S)=1$. Moreover, let us simply set $\alpha=\alpha(K)$. By Lemma~\ref{lem.min} there exists an equilateral triangle $E$ of width $\omega(E)=1$ such that $d_{\mathcal H}(K,E)=\alpha$, which implies $K \subset (E)_\alpha$ and $E \subset (K)_\alpha$.

We prove the desired estimate according to the size of $\alpha$, for which the $\alpha$-retraction $(E)_{-\alpha}$ is or not the empty set. 

\smallskip

\emph{($\alpha$ large) $\alpha\geq 1/3$.} 
By construction, the two inclusions for $K$ imply $K\subset  (E)_\alpha\cap S\subset R_\alpha$ for a suitable rectangle $R_\alpha$ of side lengths $\mathrm{diam}(E)_\alpha$ and $\omega(S)=1$. By the triangle inequality, for every $x,y\in (E)_\alpha$, we have
$$
\mathrm{dist}(x,y) \leq \mathrm{dist}(x,E) + \mathrm{diam}(E) + \mathrm{dist}(y,E) \leq {2\sqrt{3}}/{3}+2\alpha,
$$
and the equality is reached by the endpoints of the longest segment in $(E)_\alpha$ containing a side of $E$. By the arbitrariness of $x,y\in (E)_\alpha$, this implies  $\mathrm{diam}(E)_\alpha = {2\sqrt{3}}/{3}+2\alpha$. Then
$$
|K \setminus E| \leq |K| \leq |R_\alpha|={\frac{2\sqrt{3}}{3}+2\alpha}.
$$
Moreover, $|E\setminus K| \leq |E| = {1}/{\sqrt{3}}$, which combined with the previous inequality and the fact that $\alpha>1/3$ gives
$$
\mathcal A(K) \leq |K \Delta E| \leq \frac{1}{\sqrt{3}} + \frac{2\sqrt{3}}{3} + 2 \alpha = \alpha \left(\frac{
\sqrt{3}}{{3}\alpha} + \frac{2\sqrt{3}}{3\alpha} + 2\right) \leq \left(3\sqrt{3} +2 \right) \alpha(K). 
$$
This is an estimate for $\alpha$ large.

\smallskip

\emph{($\alpha$ small) $\alpha< 1/3$.} From the inclusion $E\subset (K)_\alpha$ we have $(E)_{-\alpha}\subset((K)_\alpha)_{-\alpha}= K$ and then $|E\setminus K|\leq |E\setminus (E)_{-\alpha}|=|E|-|(E)_{-\alpha}|$. By using \eqref{contE} to expand $(E)_{-\alpha}$ it follows that
$$|E\setminus K|\leq \left(1-\left(1-\frac{\alpha}{r(E)}\right)^2\right)|E|
=\left(2-\frac{\alpha}{r(E)}\right) \frac{|E|}{r(E)}\alpha.$$
From the inclusion $K\subset (E)_\alpha$ we deduce a similar inequality: using \eqref{contE}, we obtain
\[
|K\setminus E|\leq |(E)_\alpha| - |E| \leq \left(\left (1 + \frac{\alpha}{r(E)} \right)^2 - 1\right)|E| = \left(2+\frac{\alpha}{r(E)}\right) \frac{|E|}{r(E)}\alpha.
\]
Therefore, since $1/r(E)=3$ and $|E|=1/\sqrt{3}$, we obtain
\[
\mathcal A(K)\leq |K \Delta E|\leq 4  \frac{|E|}{r(E)}\alpha = 4 \sqrt{3} \alpha.
\]
This concludes the proof, as the constant obtained for $\alpha$ large is greater than that for $\alpha$ small.
\end{proof}

\section{Proof of the quantitative P\'al inequality via the inradius}\label{sec-fromator}
As a consequence of Lemma~\ref{lemma-contact} one can derive well-known inequalities relating the inradius and the width of a convex body (see, e.g., \cite[Exercise 6-2]{yagbol}).

\begin{theorem}\label{teo.rw}
Let $K$ be a planar convex body. Then
\begin{equation}\label{ratio-rw}
\frac{r(K)}{\omega(K)} \leq  \frac12
\end{equation}
and the equality holds if and only if $K$ has contact points on the same diameter (i.e., $K$ is a disk, a rectangle, \dots).

If $K$ is such that $r(K)< \omega(K)/2$, letting $T$ be the triangle circumscribed to $K$ of sides $s_1\geq s_2\geq s_3$, there holds
\begin{equation}\label{ratio-rw2better}
\frac{r(K)}{w(K)} - \frac13 \geq \frac{r(K) (\omega(T)-\omega(K)) }{\omega(K)\omega(T)}+ \frac{(s_1-s_3)}{3(s_1+s_2+s_3)}.
\end{equation}
In particular 
\begin{equation}\label{ratio-rw2}
\frac{r(K)}{\omega(K)} \geq \frac13,
\end{equation}
and the equality holds if and only if $K$ is an equilateral triangle.
\end{theorem}
\begin{proof}
To fix the notation we will need in the following, let us recall the simple but clever idea uniderlying the proof. For a planar convex set $K$ let $D$ be an indisk of $K$. The inclusion $D\subset K$ gives $2r(D)=w(D)\leq w(K)$, that is \eqref{ratio-rw}. It is clear that the equality  $\omega(K)=2r(K)$ holds if and only if $K$ has two contact points on a diameter of $D$.

Now if $K$ is such that $r(K)< \omega(K)/2$ then, by Lemma \ref{lemma-contact}, there exists a circumscribed triangle $T$. Let $r=r(K)$, $\omega=\omega(K)$, $h_1=\omega(T)$. By \eqref{wrkt} $r=r(T)$ and $\omega\leq h_1$. 
In order to give a lower bound for $r/\omega - 1/3$, we split the difference as follows:
\begin{equation}\label{split}
\frac{r}{\omega} - \frac13 = \left(  \frac{r}{\omega} - \frac{r}{h_1}   \right) + \left( \frac{r}{h_1 } - \frac13 \right).
\end{equation}
Let us write $r/h_1$ in a more convenient way. The area of the triangle $T$ satisfies
\begin{equation}\label{area1}
2|T| = {s_1}h_1= (s_1+s_2+s_3)r.
\end{equation}
The former equality is obtained by considering the base $s_1$. The latter equality is obtained by joining the three vertexes of $T$ with the center of the indisk $D$, namely by dividing $T$ into 3 triangles having base $s_i$ and height $r$. Formula \eqref{area1} allows us to deduce that
\begin{equation*}
\frac{r}{h_1} = \frac{s_1}{s_1+s_2+s_3},
\end{equation*}
in particular
$$
\frac{r}{h_1}-\frac13 = \frac{(s_1 - s_2) + (s_1 - s_3)}{3(s_1+s_2+s_3)}.
$$

Going back to \eqref{split}, we obtain
$$
\frac{r}{\omega}-\frac13 = \frac{r (h_1-\omega)}{\omega h_1}+ \frac{(s_1-s_2) + (s_1-s_3)}{3(s_1+s_2+s_3)}\geq  \frac{r (h_1-\omega) }{\omega h_1}+ \frac{ (s_1-s_3)}{3(s_1+s_2+s_3)} \geq 0.
$$
The two inequalities simply follow from $h_1 \geq \omega$ and $s_1-s_3\geq 0$. If equality holds in \eqref{ratio-rw2} then  $s_1=s_2=s_3$, i.e., the circumscribed triangle is equilateral; moreover, $K$ and $T$ must satisfy $h_1=\omega$, implying that $K$ is an equilateral triangle and $K=T$. The fact that equilateral triangles satisfy \eqref{ratio-rw2} is trivial.
\end{proof}

In view of Theorem~\ref{teo.rw} the ratio $r/\omega$ varies in the interval $[1/3, 1/2]$. For this reason we consider the following function just defined on that interval.
\begin{lemma}\label{properties}
The function $\varphi\colon [1/3,1/2]\to \mathbb [0,+\infty)$ defined by 
\begin{equation}\label{phi}
\varphi(x):=3x^2\left(\frac{\pi}{3}-\arccos\left(\frac{x}{1-x}\right)\right) + 3x \sqrt{1-2x} 
\end{equation}
is strictly increasing and thus, for all $x \in [1/3,1/2]$, 
$$
\varphi(x) \geq \varphi(1/3)= \frac{1}{\sqrt{3}}.
$$
Moreover, for all $x\in [1/3,1/2]$, 
\begin{equation}\label{Lip}
\varphi(x)-\varphi(1/3) \geq \frac{1}{\sqrt{5}}(x-1/3).
\end{equation}
\end{lemma}
\begin{proof}
The function $\varphi$ is clearly continuous in $[1/3,1/2]$ and differentiable in $(1/3,1/2)$. A direct computation of the derivative gives
\[
\varphi'(x) 
=  6 x \left[ \frac{\pi}{3}  - \arccos\left(\frac{x}{1-x}\right)\right] + \frac{3 (1-2x)^{3/2}}{1-x}:=f(x)+g(x).
\]
We have written $\varphi'$ as the sum of two positive functions, that we call $f$ and $g$. Indeed, in $(1/3,1/2)$, the function $x/(1-x)$ is strictly increasing and thus $\arccos(x/(1-x))$ is strictly decreasing with values ranging from $\pi/3$ to $0$, which implies the positivity of $f$. The positivity of $g$ is obvious. Since $\varphi'>0$ in $(1/3,1/2)$ then $\varphi$ is striclty increasing in $[1/3,1/2]$ with minimum in $x=1/3$, see Figure \ref{fig4}.

Now, by the mean value theorem, for every $x\in (1/3,1/2)$ there exists $\xi_x\in (1/3,x)$ such that
$$
\varphi(x)-\varphi(1/3) = (x-1/3) \varphi'(\xi_x) \geq (x-1/3) \min_{\xi \in [1/3,1/2]}\varphi'(\xi).
$$
The minimum of $\varphi'$ can be obtained by imposing $\varphi''(x)=0$, but this leads to a trigonometric equation that cannot be explicitly solved. Therefore, we just look for an estimate of this minimum.
As already noticed, $f,g\geq0$. The function $f$ is increasing because it is the product of two increasing functions. As for the function $g$, it is decreasing, as it can be seen that
$$
g'(x) =\frac{3(x-2)\sqrt{1-2x}}{(1-x)^2}< 0 \quad \hbox{in }(1/3,1/2),
$$
see Figure \ref{fig5}.
Let $x_0$ be a value (that we will choose in a moment) in the interval $(1/3,1/2)$. For all $x\in [1/3,x_0]$ we have $\varphi'(x)=f(x)+g(x) \geq g(x) \geq g(x_0)$, and for all $x\in [x_0,1/2]$
we have $\varphi'(x)=f(x)+g(x) \geq f(x) \geq f(x_0)$. Therefore, for all $x\in[1/3,1/2]$ we obtain $\varphi'(x) \geq \min\{f(x_0), g(x_0)\}$ and by choosing $x_0=2/5$
$$
\varphi'(x) \geq \min\{g(3/8), f(3/8)\}=\min\{1/\sqrt5, 12(\pi/3-\arccos(2/3))/5\}= {1}/{{\sqrt5}}.
$$
This gives the desired estimate.
\end{proof}
\begin{figure}
\centering
\begin{minipage}{0.5\textwidth}
\begin{tikzpicture}[scale=0.7]
\begin{axis}[
    axis lines = middle,
    xlabel = $x$,
    xlabel style={xshift=10ex}, 
    ymin=0.3, ymax=1.5,
    xmin=0.3, xmax=0.5,
    samples=100,
    domain=0.3:0.5
]
\addplot [
    domain=1/3:0.5, 
    samples=100, 
    color=black,
    thick,
]
{pi*x^2-3*x^2*rad(acos(x/(1-x)))+3*x*sqrt(1-2*x)};
\addplot [
    domain=1/3:0.5, 
    samples=100, 
    color=black,
   ]
{6*x*(pi/3-rad(acos(x/(1-x))))+(3*(1-2*x)^(3/2))/(1-x)};
\node[above right] at (axis cs: 0.44,0.57) {$\varphi(x)$};
\node[above right] at (axis cs: 0.44,1.15) {$\varphi'(x)$};
\end{axis}
\end{tikzpicture}
\caption{The function $\varphi$ and its derivative $\varphi'$.}\label{fig4}
 \end{minipage}\hfill
 \begin{minipage}{0.5\textwidth}
        \centering
\quad
\begin{tikzpicture}[scale=0.7]
\begin{axis}[
    axis lines = middle,
    xlabel = $x$,
    xlabel style={xshift=10ex}, 
    ymin=0.3, ymax=1.5,
    xmin=0.3, xmax=0.5,
    samples=100,
    domain=0.3:0.5
]
\addplot [
    domain=1/3:0.5, 
    samples=100, 
    color=black,
    thick,
]
{6*x*(pi/3-rad(acos(x/(1-x))))};
\addplot [
    domain=1/3:0.5, 
    samples=100, 
    color=black,
   ]
{(3*(1-2*x)^(3/2))/(1-x)};
\node[above right] at (axis cs: 0.32,0.87) {$g(x)$};
\node[above right] at (axis cs: 0.43,1.35) {$f(x)$};
\end{axis}
\end{tikzpicture}
\caption{The functions $f$ and $g$.}\label{fig5}
\end{minipage}
\end{figure}

The next  result is crucial and relates the area to the inradius, allowing for a partial reduction of the quantitative inequality for the area to a quantitative one for the inradius, always under width constraint. We discuss this issue in a moment.

\begin{proposition}\label{lowerphi}
For every planar convex body $K$ we have
$$
\frac{|K|}{\omega^2(K)} \geq \varphi\left( \frac{r(K)}{\omega(K)} \right),
$$
where the function $\varphi$ is defined in \eqref{phi}.
\end{proposition}

\begin{figure}
\centering

\begin{tikzpicture}[scale=0.7]
\coordinate (O) at (0,0);  
\def\radius{3}  
\def\a{35}
\draw[thick] (O) circle (\radius);
\coordinate (Q1') at ({\radius/cos(\a)}, 0);  
\coordinate (Q1) at ({\radius/cos(\a)+1.2}, 0);  
\coordinate (A) at ({\radius*cos(\a)}, {\radius*sin(\a)});  
\coordinate (B) at ({\radius*cos(\a)}, {-\radius*sin(\a)});  
 \fill[gray!30] (A) -- (B) -- (Q1') -- cycle;
    \fill[white] (0,0) circle(3); 
    
\fill (O) circle (2pt) node[below left] {$O$};
\fill (Q1') circle (2pt) node[below right] {$Q_1'$};
\fill (Q1) circle (2pt) node[below right] {$Q_1$};
\fill  (A) circle (2pt) node[above right] {$A$};
\fill (B) circle (2pt) node[below right] {$B$};

\draw pic[draw=black, angle eccentricity=1.2, angle radius=5mm]  {angle=Q1--O--A};
\node[above right] at (0.7,0.04) {$\theta$};

\node[above left] at (-2.2,2.2) {$D$};
\node[below left] at (2.8,0) {$\omega-r$};
\node[above left] at (1,1) {$r$};
\node[above right] at (2.8,0.8) {$\sqrt{(\omega-r)^2-r^2}$};
\node at (3.2,0.3) {$L$};

\draw[thick] (Q1) -- (O);
\draw[thick] (A) -- (O);
\draw[thick] (Q1') -- (A); 
\draw[thick] (Q1') -- (B); 

  \draw[thin] (A) -- ++(0.14,-0.2) -- ++(-0.2,-0.14) -- ++(-0.14,0.2) -- cycle;

\end{tikzpicture}
\caption{The construction of the ``curved triangle'' $L$ (in grey) used in the proof of Proposition~\ref{lowerphi}.}\label{fig6}
\end{figure}

\begin{proof}
By Lemma \ref{lemma-contact} two situations may occur.
In case (i) we have $\omega(K)=2r(K)$ and by taking an indisk $D\subset K$, we have $|K| \geq |D|=\pi r^2(D)=\pi r^2(K)$. Then
$$
\frac{|K|}{\omega^2(K)} \geq \frac{\pi r^2(K)}{4 r^2(K)} = \frac{\pi}{4} = \varphi\left( \frac12\right) = \varphi\left(\frac{r(K)}{\omega(K)}\right).
$$
Let us now focus on the case (ii). Let us use the notation introduced in Section~\ref{sec-prel}, see also Figure \ref{fig.KDT}.  In particular, we use $r$ for $r(K)$ and $\omega$ for $\omega(K)$.
Consider the convex hull $C:=\mathrm{conv}(D; Q_1;Q_2; Q_3)$. By convexity, we infer that $C$ is contained into $K$, in particular $|K| \geq |C|$. We now construct a smaller set $C^\prime\subset C$, whose area can be bounded from below in an easy way. 

Let $O$ denote the center of the indisk $D$. We call $\tilde{\ell}_i$ the supporting line parallel to $\ell_i$ at distance $\omega_i$ from $\ell_i$ (these correspond to the dashed lines in Figure \ref{fig.KDT}). In particular $Q_i \in \tilde{\ell}_i$. Then, for every $1\leq i \leq 3$, there holds
$$
\mathrm{dist}(Q_i,O) \geq \mathrm{dist}(\tilde{\ell}_i,O) = \omega_i - \mathrm{dist}(\ell_i,O) = \omega_i - r \geq \omega-r. 
$$
Here we have used the fact that $\ell_i$ is parallel to $\tilde{\ell}_i$ and tangent to $D$. Let ${Q}^\prime_i$ denote the point of the segment $Q_iO$, possibly coinciding with $Q_i$, which satisfies $\mathrm{dist}({Q}^\prime_i,O)=\omega-r$, see Figure~\ref{fig6}. By convexity, since $Q_i$ and $O$ belong to $K$, we infer that also ${Q}_i^\prime$ belongs to $K$. In particular ${C}^\prime:=\mathrm{conv}(D;{Q}^\prime_1;{Q}^\prime_2;{Q}^\prime_3)\subset K$ and $|K|\geq |{C}^\prime|$. The advantage of this operation is that the area of ${C}^\prime$ only depends on $\omega$ and $r$, whereas the area of $C$ depends on $\omega_i$ and the position of $Q_i$ on $\tilde{\ell}_i$. Let us compute $|{C}^\prime|$: the set ${C}^\prime$ is the union of $D$ and three copies (rotations) of the same set $L$. The set $L$ is a "curved" triangle: its boundary is made of a boundary arc of $\partial D$ and two segments starting from ${Q}^\prime_1$ and touching $\partial D$ tangentially. Let $A$ and $B$ denote the endpoints of the two segments passing through ${Q}^\prime_1$ tangent to $D$. Then the triangle $OA{Q}^\prime_1$ (and its homothetic copy ${Q}^\prime_1BO$) is rectangle. Therefore the segment $A{Q}^\prime_1$ has length $\sqrt{(\omega-r)^2-r^2}$ and the angle in $O$ is
$$
\theta:=\widehat{AO{Q}^\prime_1} = \arccos(r/(\omega-r)).
$$
The area of $L$ is nothing but the difference of twice the area of the triangle $AO{Q}^\prime_1$ and twice the area of the circular sector of angle $\theta$ in $O$:
$$
|L| = r \sqrt{(\omega-r)^2-r^2} -  \theta r^2 .
$$
Using this expression, we get
$$
|{C}^\prime|= |D| + 3|L| = \pi r^2 + 3  r \sqrt{(\omega-r)^2-r^2} -  3 \theta r^2 .
$$
Dividing by $\omega^2$, we obtain the desired inequality:
$$
\frac{|K|}{\omega^2} \geq \frac{|{C}^\prime|}{\omega^2} \geq \pi \left(\frac{r}{\omega}\right)^2 + 3  \frac{r}{\omega} \sqrt{1-2 \frac{r}{\omega}} -  3 \theta \left( \frac{r}{\omega}\right)^2 .
$$
This proves the lemma.
\end{proof}

We are now in a position to prove the first quantitative version.

\begin{proof}[Proof of Theorem~\ref{teo1}]
The estimate follows simply by combining the results in Proposition~\ref{lowerphi} and Lemma \ref{properties}, in particular by using \eqref{Lip}. This also gives the constant $c_1=1/\sqrt{5}$.
\end{proof}

\section{Intermezzo: a quantitative inequality for the inradius} \label{sec-r}

A key step in the proof of the P\'al inequality is an estimate involving the inradius under width constraint. 
In the proof of Theorem~\ref{teo2}, we rely on a quantitative version of this estimate. Given its novelty and standalone significance, we dedicate an entire section to this estimate. 
Under width constraint, there is no hope to bound the quantity $\eta$ in terms of the asymmetries $\alpha$ or $\mathcal A$.
This can be seen, for instance, by considering a sequence of rectangles of longest sides diverging towards infinity: along the sequence $\eta$ is constant, but the asymmetries $\alpha$ and $\mathcal A$ diverge to infinity. 
This is due to the fact that the inradius is a too rough quantity to describe the shape: this is evident by considering a disk $D$ and a rectangle $R$, having the same deficit but very different asymmetries from an equilateral triangle.

Consequently, we are compelled to adopt a weaker notion of asymmetry. We truncate the asymmetry $\alpha$ at the value corresponding to the asymmetry of the disk $\alpha(D)=1/6$ (see Lemma~\ref{unsesto}), that is, we consider
\begin{equation}\label{beta}
\beta_E(K):=\min\{\alpha(K),\alpha(D)\}=\min\{\alpha(K),1/6\}.
\end{equation}

\begin{theorem}\label{r-omega-quantitative}
There exists a constant $c_4>0$ such that, for every planar convex body $K$ there holds
$$
\frac{r(K)}{\omega(K)}-\frac{1}{3} \geq c_4 \beta_E(K),
$$
with $\beta_E$ as in \eqref{beta}.
\end{theorem}

\begin{proof}
Let $K$ be a planar convex body and $D$ be an indisk of $K$. For brevity denote by $r=r(K)=r(D)$ its inradius, by $\omega=\omega(K)$ its width and by $\eta:=r/\omega - 1/3$ its deficit. From Theorem~\ref{ratio-rw} the deficit $\eta\in[0,1/6]$. We prove the desired estimate according to the size of $\eta$. Fix some $0 < \delta \leq 1/2$.

\smallskip

\emph{(Large inradius).} If $\eta\geq \delta/3$ then the quantitative estimate easily follows with the constant $2\delta$. 
Indeed, by Lemma \ref{unsesto} there holds
$$
\eta(K) \geq \frac{\delta}{3} =  2\delta \cdot \alpha(D) \geq 2\delta \cdot \min\{\alpha(K), \alpha(D)\} = 2\delta \cdot \beta(K).
$$

\smallskip

\emph{(Small inradius).} For the rest of the proof, we can thus assume $\eta< \delta/3$, that is 
\begin{equation}\label{hp1}
r<\left(\frac{1+\delta}{3}\right)\omega.
\end{equation}  
By \eqref{ratio-rw} there exists the circumscribed triangle $T$ to $D$ and $K$. 
Since the sides $s_1,s_2,s_3$ of the triangle $T$ satisfy $s_1 h_1 = 2|T| = r (s_1+s_2+s_3)$, we deduce that 
\begin{equation}\label{clever2}
h_1 = r (s_1+s_2+s_3)/s_1.
\end{equation}
By combining \eqref{clever2} with the assumption \eqref{hp1} and $s_1\geq s_2\geq s_3$, we obtain
\begin{equation}\label{stimah}
h_1\leq (1+\delta)\omega.
\end{equation}
On the other hand, using again \eqref{clever2} and the inequality \eqref{hp1} combined with $\omega\leq h_1 $, we infer that $3s_1\leq (1+\delta)(s_1+s_2+s_3)$. This last inequality, combined either with $s_3\leq s_2$ or with $s_2\leq s_1$, allows to estimate the longest side $s_1$ in terms of the shorter ones $s_2$ and $s_3$:
\begin{equation}\label{stimas}
s_1\leq \frac{1+\delta}{1-\delta/2}s_2\quad\text{and}\quad s_1\leq \frac{1+\delta}{1-2\delta}s_3. 
\end{equation}
The inequalities \eqref{stimah} and \eqref{stimas} mean that the smaller $\delta$, the closer $T$ to an equilateral triangle. 
Now, by using the sine formula for $T$ and recalling that the largest angle $\gamma_1$ of $T$ satisfies $\gamma_1\geq \pi/3$, we deduce that
\begin{equation}\label{sin}
h_1=s_2\sin\gamma_3 =  \frac{s_2s_3}{s_1}\sin\gamma_1\geq \frac{s_2s_3}{s_1}\frac{\sqrt{3}}{{2}}.
\end{equation}
By \eqref{stimah} and \eqref{sin}, together with \eqref{stimas}, we thus have
\begin{equation}\label{fund}
s_2\leq \frac{2}{\sqrt{3}} \frac{(1+\delta)^2}{1-2\delta}\omega \quad \hbox{and} \quad 
s_3\leq \frac{2}{\sqrt{3}} \frac{(1+\delta)^2}{1-\delta/2}\omega.
\end{equation}
By using the first inequality \eqref{stimas} and \eqref{fund}, we estimate the perimeter of $T$ in terms of the width of $K$:
\begin{equation}\label{111}
s_1+s_2+s_3\leq \frac{1+\delta}{1-\delta/2}s_2 + s_2 + s_3 \leq 2\sqrt{3} \frac{(1+\delta)^2}{(1-2\delta)}\omega.
\end{equation}

Now, in view of Proposition~\ref{asy.convex}, recalling that $\gamma_T=\gamma_3$ and writing $h_1$ in terms of $s_2$ as $h_1=s_2\cos\gamma_3\tan\gamma_3$, 
we have 
\begin{equation}\label{stimaa}
\alpha(K)\leq \frac{(s_1-s_3)}{\omega} + \frac{s_3}{\omega}\frac{(h_1-\omega)}{h_1}+\frac{\sqrt{3}s_2\cos \gamma_3-\omega}{\omega}.
\end{equation}
We estimate each term in the sum above as follows. We use \eqref{111} to estimate the first term
\begin{equation}\label{la1}
\frac{(s_1-s_3)}{\omega} \leq \frac{6\sqrt{3}(1+\delta)^2}{(1-2\delta)}\frac{(s_1-s_3)}{3(s_1+s_2+s_3)}.
\end{equation}
For the second one we use the second inequality in \eqref{fund} and then $1 \leq 3r/\omega$ (see \eqref{ratio-rw2}), to obtain
\begin{equation}\label{la2}
 \frac{s_3}{\omega}\frac{(h_1-\omega)}{h_1}\leq 
  \frac{2\sqrt{3}{(1+\delta)^2}}{1-\delta/2}\frac{r(h_1-\omega)}{\omega h_1}.
\end{equation}
The third term can be rewritten in a more convenient way
\begin{align*}
\frac{\sqrt{3}s_2\cos \gamma_3-\omega}{\omega} & = \frac{s_2}{\omega}\left( \sqrt{3}\cos \gamma_3-  \sin\gamma_3\right) + \frac{h_1-\omega}{\omega} 
\\
& = \frac{s_2}{\omega(\sqrt{3}\cos\gamma_3 + \sin \gamma_3)}\left({4}\cos^2 \gamma_3- 1\right) + \frac{h_1-\omega}{\omega}
\end{align*}
so that by \eqref{fund}, $\sqrt{3}\cos\gamma_3 + \sin\gamma_3\geq \sqrt{3}$ and $h_1 \leq 3r$ (see \eqref{ratio-rw2} applied to $T$), 
\begin{equation}\label{333}
\frac{\sqrt{3}s_2\cos \gamma_3-\omega}{\omega} \leq \frac{2}{3}\frac{(1+\delta)^2}{(1-2\delta)}\left({4}\cos^2 \gamma_3- 1\right) + 3 \frac{r(h_1-\omega)}{h_1\omega}.
\end{equation}
To estimate the remaining term $4\cos^2\gamma_3-1$, we exploit the cosine formula to obtain
\[
4 \cos^2(\gamma_3) - 1 = \frac{(s_1^2+s_2^2-s_3^2)^2}{s_1^2s_2^2} - 1  = \frac{s_1^4+s_2^4+s_3^4+s_1^2s_2^2-2s_1^2s_3^2-2s_2^2s_3^2}{s_1^2s_2^2},
\]
which combined with the relation $s_1\geq s_2 \geq s_3$ and \eqref{stimas} yields
\[
4 \cos^2(\gamma_3) - 1 \leq 4\frac{s_1^4-s_3^4}{s_1^2s_2^2}\leq \frac{16(s_1-s_3)s_1}{s_2^2}\leq \frac{16(s_1-s_3)s_1}{s_2^2}\leq\frac{16(1+\delta)}{(1-\delta/2)}\frac{(s_1-s_3)}{s_2}.
\]
By using the first inequality of \eqref{111} together with $s_3 \leq s_2$, from the previous estimate we obtain
\[
4 \cos^2(\gamma_3) - 1 \leq \frac{16(1+\delta)(s_1-s_3)}{3(s_1+s_2+s_3)},
\]
which plugged into \eqref{333} gives an estimate for the first term in the previous sum
\begin{equation}\label{la3}
\frac{\sqrt{3}s_2\cos \gamma_3-\omega}{\omega} \leq \frac{32}{3}\frac{(1+\delta)^3}{(1-2\delta)} \frac{(s_1-s_3)}{3(s_1+s_2+s_3)}+ 3 \frac{r(h_1-\omega)}{h_1\omega}.
\end{equation}

Finally, by plugging \eqref{la1}, \eqref{la2}, and \eqref{la3} in \eqref{stimaa}, we infer that
$$
\alpha(K)  \leq b_1(\delta)  \frac{(s_1-s_3)}{3(s_1+s_2+s_3)}  +  b_2(\delta)\frac{r(h_1-\omega)}{h_1\omega},
$$
which combined with \eqref{ratio-rw2better} implies
\begin{equation}\label{finale}
\eta(K) \geq \frac{1}{\max\{b_1(\delta), b_2(\delta)\}} \alpha(K),
\end{equation}
where 
\[
b_1(\delta):=\frac{(1+\delta)^2(6\sqrt{3}+\frac{32}{3}(1+\delta))}{1-2\delta} \ \text{ and } \ b_2(\delta):=\left(\frac{2\sqrt{3}{(1+\delta)^2}}{1-\delta/2}+3\right).
\]
This is the quantitative estimate in the small inradius regime, thanks to the fact that $\alpha(K) \geq \beta(K)$. Some comments about the value of the constant $c_4$ appearing in the statement are in the following Remark \ref{c4-rmk}.
\end{proof}

\begin{remark}\label{c4-rmk} In view of the previous proof, the constant $c_4$ appearing in the statement of Theorem \ref{r-omega-quantitative} can be taken as  
$$
c_4(\delta)=\min \left\{ 2 \delta,  \frac{1}{b_1(\delta)}, \frac{1}{b_2(\delta)}\right\},
$$
for some $0<\delta\leq 1/2$. A good choice is $\delta = 1/50$ that gives $c_4= 1/25$.
\end{remark}

\section{Proofs of the quantitative P\'al inequality via asymmetries}\label{sec-a}
In this section we prove the quantitative versions of the P\`al inequality in terms of the asymmetry $\alpha$, defined via the Hausdorff distance, and, as a consequence of Proposition~\ref{equivalence1}, we derive another one for the Fraenkel asymmetry $\mathcal A$.

We start with a refined version of Proposition \ref{lowerphi} which involves, besides the inradius, a sort of circumradius $m(K)$. More precisely, given an indisk $D$ of $K$, assumed without loss of generality centered at the origin, we set
\begin{equation}\label{M}
m(K):=\max_{P\in K} \|P\|.
\end{equation}
Of course this quantity depends on the indisk chosen (and in particular it depends on the center of $D$).

\begin{lemma}\label{psi} 
The function $\psi:\{(x,y)\in \mathbb R^2 : 1/3 \leq x \leq 1/2\,,\ y\geq 1-x\}\to [0,+\infty)$ defined by
\begin{equation}\label{psi2}
\psi(x,y):=\pi x^2 + 2  x\sqrt{1-2x}- 2 x^2 \arccos\frac{x}{1-x}  + x \sqrt{y^2-x^2} - x^2 \arccos\frac{x}{y}
\end{equation}
is increasing in both variables, and in particular 
\begin{equation*}
\psi(x,y) \geq \frac{1}{\sqrt{3}}.
\end{equation*}
\end{lemma}

\begin{proof}
The partial derivatives ${\partial \psi}/{\partial y}$ is equal to $({x}/{y})\sqrt{y^2-x^2}$, which is clearly non-negative on the domain of definition of $\psi$. The partial derivative with respect to $x$ is
\begin{equation*}
\frac{\partial \psi}{\partial x} = 
2\pi x + 2 \frac{(1-2x)^{3/2}}{(1-x)}  - 4 x \arccos(x/(1-x))
+ \sqrt{y^2-x^2} - 2x \arccos(x/y),
\end{equation*}
and thus it is more tricky to estimate. 
By the monotonicity of the maps $y\mapsto \sqrt{y^2-x^2}$ and 
$y\mapsto \arccos(x/y)$, we infer that 
$\sqrt{y^2-x^2}\geq \sqrt{1-2x}$ and $-2x\arccos(x/y) \geq - \pi x$. By using these estimates for the last two terms in the expression of ${\partial \psi}/{\partial x}$, and by getting rid of the second term, we obtain the estimate ${\partial \psi}/{\partial x}\geq h$ with the function $h\colon (1/3,1/2)\to \mathbb R$ defined as 
\[
h(x):=\pi x-4x\arccos(x/(1-x))+\sqrt{1-2x}.
\]
We aim at proving the positivity of $h$ in $(1/3,1/2)$. To do so, we first notice that
\[
h(1/3)=\frac{1}{3}\left(\sqrt{3}-\frac{\pi}{3}\right)>0 \quad \text{and} \quad h'(1/3)=\sqrt{3}-\frac{\pi}{3}>0.
\]
Then, since 
\[
h''(x)=\frac{9(1-2x)(1-x)+(5x-1)(2-3x)}{(1-2x)^{3/2}(1-x)^2}>0
\] 
for every $x\in(1/3,1/2)$, $h$ is convex in $(1/3,1/2)$, and so $h'$ is increasing with $h'(x)\geq h'(1/3)>0$ for every $x\in(1/3,1/2)$. This in turn says that $h$ is increasing with $h(x)\geq h(1/3)>0$ for every $x\in(1/3,1/2)$, providing the positivity of $h$ in $(1/3,1/2)$. As a consequence of the previous lower bound it follows the positivity of $\partial \psi/\partial x$
for every admissible $x$ and $y$. The positivity of the partial derivatives yields the monotonicity of $\psi$ in both variables.
Moreover, $y\geq 1-x$ yields $\psi(x,y)\geq \psi(x,1-x)$. Comparing the definition \eqref{psi2} of $\psi$ with the definition \eqref{phi} of $\varphi$, it is immediate to check that $\psi(x,1-x)=\varphi(x)$. By Lemma \ref{properties} $\varphi(x) \geq 1/\sqrt{3}$ then $\psi(x,1-x)\geq 1/\sqrt{3}$, which concludes the proof.
\end{proof}

\begin{proposition}\label{lowerpsi}
For every planar convex body $K$ we have
\begin{equation}\label{lowerpsi2}
\frac{|K|}{\omega^2(K)} \geq \psi\left( \frac{r(K)}{\omega(K)}, \frac{m(K)}{\omega(K)}\right),
\end{equation}
where the function $\psi$ is defined in \eqref{psi} and $m$ is defined in \eqref{M}.
\end{proposition}
\begin{proof}
Let $D$ be an indisk of $K$, without loss of generality centered at the origin. For brevity, denote by $r=r(K)=r(D)$, $\omega=\omega(K)$, and $m=m(K)$. Let $V\in K$ be a point of maximal norm, namely such that $m=\|V\|$. 

Assume first $r<\omega / 2$. Then, in view of Lemma \ref{lemma-contact}, there exist three contact points $P_1, P_2, P_3\in \partial D\cap \partial K$ forming an acute-angled triangle. Connecting the points $P_i$ with the origin, we split $K$ into three regions $K_1, K_2, K_3$. Without loss of generality, we may assume that $V\in K_1$. As already done in the proof of Proposition~\ref{lowerphi}, for $i=2,3$ we find ${Q}_i^\prime\in K_i$ with $\|Q_i^\prime\|=\omega-r$, see Figure~\ref{fig6}. The convex hull $C:=\mathrm{conv}(D;V;Q_2^\prime; Q_3^\prime)$ is, by construction and by convexity of $K$, a subset of $K$. In particular $|K|\geq |C|$. The area of $C$ can be easily computed, by noticing that $C$ is the union of $D$ and three "curved triangles". The boundary of each curved triangle is made of an arc of $\partial D$ and two segments starting from $V$ or ${Q}_2^\prime$ or $Q_3^\prime$, tangent to $\partial D$ at the endpoints of the arc. It is immediate to check that the area of the curved triangle starting from a point $P$, $P\notin D$, is
$$
r \sqrt{\|P\|^2-r^2} - r^2\arccos(r/\|P\|).
$$
Applying this formula to $P=V$, $P=Q_i^\prime$, and recalling that $\|V\|=m$, $\|Q_i^\prime \| = \omega-r$, we obtain
\begin{equation*}
|K|\geq |C| = \pi r^2 + r \sqrt{m^2-r^2} - r^2\arccos\frac rm
+2\Big( r\sqrt{(\omega-r)^2-r^2} - r^2 \arccos\frac{r}{\omega-r}\Big).
\end{equation*} 
Dividing this expression by $\omega^2$ and recalling definition \eqref{psi2} of $\psi$, we obtain the estimate \eqref{lowerpsi2} in the regime $r<\omega/2$.

In the case $r=\omega/2$, we consider the convex hull $\widetilde{C}:=\mathrm{conv}(D;V)\subset K$, namely the union of $D$ and only 1 ``curved triangle":
$$
|K| \geq |\widetilde{C}| =\pi r^2 + r \sqrt{x^2-r^2} - r^2\arccos(r/x).
$$ 
Again, dividing by $\omega^2$ and recalling that $r=\omega/2$, one obtains \eqref{lowerpsi2}.
\end{proof}

\begin{proof}[Proof of Theorem~\ref{teo2}]
Let $K$ be a planar convex body and $D$ be an indisk of $K$. Without loss of generality, we assume $D$ to be centered at the origin. For brevity we use the notation $r=r(K)=r(D)$, $\omega=\omega(K)$, $\eta=r/\omega - 1/3$. From Theorem~\ref{ratio-rw} the quantity $\eta\in[0,1/6]$ and again, as in the proof of Theorem \ref{r-omega-quantitative}, we divide the proof according to the size of $\eta$. Fix some $0<\delta\leq 1/2$. 

\smallskip
\emph{(Small inradius).} If $\eta< \delta/3$, then the quantitative estimate easily follows from what we already proved.
Indeed by using Theorem~\ref{teo1} with \eqref{finale} (that corresponds to the small inradius case in the proof of Theorem~\ref{r-omega-quantitative}), we have 
\[
\frac{|K|}{\omega^2(K)} - \frac{1}{\sqrt{3}} \geq c_1\eta(K)\geq c_1c_4\alpha(K)
\]
and the quantitative inequality holds with the constant $c_1c_4$.

\smallskip

\emph{(Large inradius).}
Let us assume that $\eta \geq \delta/3$. In order to find a quantitative estimate relating deficit and asymmetry, our strategy consists in providing a strictly positive lower bound of the ratio deficit/asymmetry. In view of Proposition \ref{lowerpsi}, the deficit satisfies
\begin{equation}\label{numerator}
|K|/\omega^2(K) -1/\sqrt{3} \geq \psi\left( \frac{r}{\omega},\frac{m}{\omega}\right) - \frac{1}{\sqrt{3}} \geq 0,
\end{equation}
being $\psi$ the function defined in \eqref{psi2} and $m:=\max_{P\in K} \|P\|$. 
Notice that the maximal distance from the origin $m$ satisfies $m\geq \omega-r$ (see also the proof of Propositions \ref{lowerphi} and \ref{lowerpsi}).

Let us bound from above the asymmetry. To this aim, we choose the equilateral triangle $E$, with $\omega(E)=\omega$, centered at the origin and with one height oriented as the segment $OV$. Let $D(m)$ denote the disk centered at the origin with radius $m$. Exploiting the inclusions
$$
D \subset K \subset D(m),
$$
we infer that
\begin{equation}\label{combi1}
d_{\mathcal H}(K,E)  \leq \max \{ d_{\mathcal H}(D,E) , d_{\mathcal H}(D(m),E)\} 
\end{equation}
Recalling that $r\leq \omega/2$ and $m\geq \omega-r\geq \omega/2$, by using \eqref{small} and \eqref{big} in Lemma \ref{TBdist} applied to $D$ and $D(m)$, respectively, we obtain
\begin{equation}\label{combi2}
d_{\mathcal H}(D,E) = 2\omega/3 - r, \quad d_{\mathcal H}(D(m),E)
= m- \omega/3.
\end{equation}
Furthermore, using once again $m\geq \omega - r$, we get
\begin{equation}\label{combi3}
\max\{ 2\omega/3 - r, m - \omega/3 \} = m- \omega/3.
\end{equation}
By combining \eqref{combi1}, \eqref{combi2}, \eqref{combi3}, we obtain the following bound on the asymmetry:
\begin{equation}\label{denominator}
\alpha(K) \leq \frac{d_{\mathcal H}(K,E)}{\omega} \leq \frac{m}{\omega} - \frac13.
\end{equation}

The estimates \eqref{numerator} and \eqref{denominator} give
\begin{equation}\label{fraction}
\frac{|K|/\omega^2(K) -1/\sqrt{3} }{\alpha(K)} \geq \frac{\psi(r/\omega, m/\omega) - 1/\sqrt{3}}{m/\omega-1/3}.
\end{equation}
Notice that, under the assumption $\eta \geq \delta/3$, $K$ is ``far" from the equilateral triangle and $\alpha(K)\neq 0$.
The estimate \eqref{fraction} allows to obtain the desired quantitative estimate, provided that 
the function
$$
\Psi(x,y):=\frac{\psi(x,y)-1/\sqrt{3}}{y-1/3}
$$
has strictly positive infimum in the set 
\begin{equation*}
\left\{(x,y)\in \mathbb R^2\ :\ (1+\delta)/3 \leq x \leq 1/2, \quad y \geq 1-x\right\}.
\end{equation*}
The two bounds for $x=r/\omega$ and $y=m/\omega$ come from the standing assumption $\eta \geq \delta/3$ and from $m\geq \omega-r$.

By using the increasing monotonicity of $x\mapsto \psi(x,y)$ proved in Proposition \ref{lowerpsi}, we infer that $\Psi$ is increasing in $x$, too and 
$$
\Psi(x,y) \geq \Psi((1+\delta)/3,y):=q(y).
$$
We prove the monotonicity of $q$ to obtain a bound on $q$. Let $p(y):=\psi((1+\delta)/3,y)$, then the derivative
\begin{equation*}
q'(y)=\frac{(y-1/3)p'(y) - (p(y)-1/\sqrt{3})}{ (y-1/3)^2}.
\end{equation*}
The denominator is clearly positive. Let us prove that also the numerator 
$N(y):=(y-1/3)p'(y) - (p(y)-1/\sqrt{3})$ is positive by studying its monotonicity.
A direct computation leads to
$$
N'(y) =  (y-1/3) p''(y) =(y-1/3) \frac{((1+\delta)/3)^3}{y^2\sqrt{y^2-((1+\delta)/3)^2}}>0,
$$
thus $N$ is an increasing function and one obtains
\begin{equation*}
N(y)  \geq N\Big(\frac{2-\delta}{3}\Big)
\!=\! \frac{1}{\sqrt{3}}\Big(1 - \frac{1+\delta}{2-\delta} \sqrt{1-2\delta}\Big) - \frac{1}{9}(1+\delta)^2 \Big( \pi - 3 \arccos\frac{1+\delta}{2-\delta}\Big).
\end{equation*}
The two terms in brackets are positive for $\delta \in (0,1/2]$. To guarantee the positivity of $N(y)$ we need to make a choice on $\delta$: a direct computation shows that the right-hand side is positive taking, e.g., $\delta=1/50$. 
With this choice on $\delta$, we infer that $q'>0$, thus $q$ is increasing. In particular
$
\Psi(x,y)\geq q(y)\geq q({(2-\delta)}/{3}) \geq \varphi((1+\delta)/3)>0,
$
and the quantitative estimate for the deficit holds with the constant $c_2=\varphi((1+\delta)/3)$. A comment on the value of $c_2$ is postponed in Remark~\ref{onc2}.
\end{proof}

The quantitative P\'al inequality for the Fraenkel asymmetry can be thus derived from Theorem~\ref{teo2}.

\begin{proof}[Proof of Theorem~ \ref{teo3}]
This is an immediate consequence of Theorem~\ref{teo2} and Proposition~\ref{equivalence1}:
$$
\frac{|K|}{\omega^2(K) }- \frac{1}{\sqrt{3}} \geq c_2 \alpha(K) \geq \frac{c_2}{3\sqrt{3} + 2} \mathcal A(K).
$$
The constant appearing in front of $\mathcal A$ is the desired $c_3$. A comment on its value is postponed in Remark \ref{onc2}.
\end{proof}

\begin{remark}\label{onc2}
In view of the proof of Theorem \ref{teo2}, the constant $c_2$ appearing in the statement can be taken as
$$
c_2(\delta)=\min\{c_1c_4(\delta), \varphi(1+ \delta)/3) \}
$$
for some $0<\delta\leq 1/2$. The previous choice $\delta = 1/50$ in Remark~\ref{c4-rmk} yields $c_4=1/25$. Moreover, by Theorem~\ref{teo1}) the constant $c_1=1/\sqrt{5}$ and therefore
$$
c_2=\min\left\{{1}/({25\sqrt{5}}), \varphi(51/150) \right\} =  \frac{1}{25\sqrt{5}}.
$$
This also allows us to compute a possible value for the constant $c_3$ appearing in the statement of Theorem~\ref{teo3}: 
$$
c_3= \frac{c_2}{3\sqrt{3} + 2} = \frac{1}{25(3\sqrt{3} + 2)\sqrt{5}}.
$$
\end{remark}

Another interesting consequence of Theorem~\ref{teo2} is the equivalence of the two asymmetries $\alpha$ and $\mathcal A$. 
\begin{corollary}\label{equivalence}
For every planar convex body, there hold
$$
 \frac{1}{25\sqrt{5}} \alpha(K) \leq \mathcal A(K) \leq (3\sqrt{3} + 2) \alpha(K).
$$
\end{corollary}
\begin{proof}
The upper bound has been already proved in Proposition~\ref{equivalence1}. Then we just focus on the first bound.
Given a planar convex body $K$ let $E$ be an equilateral triangle with $\omega(E)=\omega(K)$. Taking the inequality
$$
|K \Delta E| = |K\setminus E| + |E\setminus K| \geq |K\setminus E| = |K|-|E| + |E\setminus K| \geq |K| - |E|,
$$
and dividing it by $\omega^2(E)$, one recognizes the deficit of the area in the right-hand side. Then, using Theorem~\ref{teo2} and taking to the minimum over the equilateral triangles $E$ of same width as of $K$, one obtains $\mathcal A(K)\geq \alpha(K)/{25\sqrt{5}}$.
\end{proof}

\section{Sharpness of the exponents}\label{sec-sharp}

Let us now prove the part concerning the sharpness of the exponent 1 appearing in the right-hand side of the quantitative inequalities in Theorems ~\ref{teo1}, \ref{teo2}, \ref{teo3}, and~\ref{r-omega-quantitative}. This means that the functional $F$ that we want to estimate from below, which in our case is either the area deficit or the quantity $\eta$, can be controlled by an asymmetry $a$ for every exponent $t\geq1$ 
\[
F(K)\geq c \cdot a^t(K)
\]
but not as $t<1$. Here and in the following $c$ will denote a generic positive constant, that can change from line to line. The usual strategy to prove this sharpness is to exhibit a sequence of sets $\left(K_\epsilon\right)_{\epsilon}$ for which in the limit as $\epsilon \to 0$ the left-hand side and the right-hand side vanish with the same rate of convergence, that is
$$\lim_{\epsilon\to 0}\frac{F(K_\epsilon)}{a(K_\epsilon)}= c$$ 
for some (finite) constant $c>0$. 

For the sharpness of the quantitative inequalities of this paper we will use the same sequence of sets, given by a family of isosceles triangles converging to an equilateral triangle. Let $E$ be the equilateral triangle with vertexes $(0,1)$, $\pm(1/\sqrt{3},0)$. The triangle $E$ has width $1$. For every $\epsilon$ sufficiently small, we consider $K_\epsilon$ the isosceles triangle with vertexes $(0,1)$ and $\pm (1/\sqrt{3} + \epsilon,0)$.

Let us start with the quantitative inequality in Theorem~\ref{teo1}, relating the area deficit (in the left-hand side) with the quantity $\eta$ (in the right-hand side).
 It is immediate to check that the width and the area of $K_\epsilon$ are given by  
\begin{equation*}
\omega(K_\epsilon)=1, \qquad |K_\epsilon|  = \frac1{\sqrt{3}} + \epsilon,
\end{equation*}
and that the perimeter of $K_\epsilon$ satisfies
\begin{equation*}
P(K_\epsilon)= 2 \sqrt{3}\left(1 +\frac{\sqrt{3}}{2} \epsilon\right) + o(\epsilon),\quad \text{as $\epsilon \to0$}.
\end{equation*}
The inradius can be computed using the relation $P(K_\epsilon) r(K_\epsilon) = 2 |K_\epsilon|$:
\[
r(K_\epsilon) = \frac13 + \frac{\epsilon}{2\sqrt{3}}+o(\epsilon),\quad \text{as $\epsilon \to0$}.
\]
Since
$$
\frac{|K_\epsilon|}{\omega(K_\epsilon)} - \frac{1}{\sqrt{3}}=\epsilon
$$
and
$$
\eta(K_\e) =  \frac{r(K_\epsilon)}{\omega(K_\epsilon)} - \frac13  = \frac{\epsilon}{2\sqrt{3}}+o(\epsilon),\quad \text{as $\epsilon \to0$},
$$
we deduce that the area deficit and the quantity $\eta$ vanish with the same rate of convergence, implying the sharpness of the exponent 1 in  Theorem~\ref{teo1}.

We proceed in the same way for the quantitative inequality appearing in Theorem \ref{teo2}, relating the area deficit (in the left-hand side) with the asymmetry $\alpha$ (in the right-hand side). We start by showing that $\alpha(K_\epsilon)=\epsilon$. Testing $\alpha$ with $E$ we obtain
$$
\alpha(K_\epsilon) \leq d_{\mathcal{H}}(K_\epsilon, E) = \epsilon.
$$
Assume by contradiction that for every $\epsilon>0$ sufficiently small this inequality is strict. This implies the existence of an equilateral triangle $\widetilde{E}$ of width $1$, such that, for some $\delta <\epsilon$,
$$
K_\epsilon \subset (\widetilde{E})_\delta.
$$
Comparing the diameters we obtain that:
$$
\frac{2}{\sqrt{3} }+ 2 \epsilon  = \mathrm{diam}(K_\epsilon) \leq  \mathrm{diam}(\widetilde{E}) + 2 \delta  \leq  \frac{2}{\sqrt{3} }+ 2 \delta,
$$
thus $\epsilon \leq \delta$, a contradiction with the hypothesis. Therefore, taking $\epsilon$ small, we infer that $\alpha$ is of order $\epsilon$, as the area deficit. This concludes the proof of the sharpness of the exponent 1 in Theorem \ref{teo2}.

The previous computations imply the sharpness of the exponent 1 also in the quantitative estimates appearing in Theorems \ref{teo3} and \ref{r-omega-quantitative}: the former relates the deficit with the asymmetry $\mathcal A$, the latter relates the quantity $\eta$ with the asymmetry $\beta$. Exploiting the equivalence \eqref{equivalence} between $\mathcal A$ and $\alpha$ and the definition \eqref{beta} of $\beta$ in terms of $\alpha$, we infer that both asymmetries vanish with rate of convergence $\epsilon$ along the sequence $K_\epsilon$.

\paragraph{\textbf{Acknowledgments.}} It is a pleasure to acknowledge Lorenzo Brasco for  pointing out the reference \cite{ind}. The first author has been partially supported by the GNAMPA 2023 project \emph{Esistenza e propriet\`a fini di forme ottime}. The authors have been partially supported by the GNAMPA 2024 project \emph{Ottimizzazione e disuguaglianze funzionali per problemi geometrico-spettrali locali e nonlocali}. The second author has been partially supported by the GNAMPA 2023 project \emph{Teoria della regolarit\`a per problemi ellittici e parabolici con diffusione anisotropa e pesata} and by the PRIN 2022 project 2022R537CS \emph{$NO^3$ - Nodal Optimization, NOnlinear elliptic
equations, NOnlocal geometric problems, with a focus on regularity}, founded
by the European Union - Next Generation EU.


\begin{thebibliography}{11}
\bibliographystyle{plain}

\bibitem{bicrhe}
Bianchini, C., Croce, G., and Henrot, A. (2023). On the quantitative isoperimetric inequality in the plane with the barycentric asymmetry. \emph{Ann. Sc. Norm. Super. Pisa Cl. Sci.} XXIV(5), 2477--2500.

\bibitem{bog}
Bogosel, B. (2023). Numerical shape optimization among convex sets. \emph{Applied Mathematics \& Optimization} 87(1).

\bibitem{bon}
Bonnesen, T. (1921). Sur une am\'elioration de l'in\'egalit\'e isop\'erimetrique du cercle et la d\'emonstration d'une in\'egalit\'e de Minkowski. \emph{CR Acad. Sci. Paris} 172, 1087--1089.

\bibitem{cacogr}
Campi, S., Colesanti, A., and Gronchi, P. (1996). Minimum problems for volumes of convex bodies, in Partial differential equations and applications, P. Marcellini, G. Talenti and E. Vesentini (eds.), Marcel Dekker, \emph{New York, U.S.A.}, 43--55.

\bibitem{came} 
Cañete, A., and Merino, B. G. (2021). On the isodiametric and isominwidth inequalities for planar bisections. \emph{Rev. Mat. Iberoam.} 37(4),1247--1275.

\bibitem{cicleo}
Cicalese, M., and Leonardi, G. P. (2012). A selection principle for the sharp quantitative isoperimetric inequality. \emph{Archive for Rational Mechanics and Analysis} 206(2), 617--643.

\bibitem{fug}
Fuglede, B. (1989). Stability in the isoperimetric problem for convex or nearly spherical domains in $\mathbb R^n$. \emph{Transactions of the American Mathematical Society} 314(2), 619--638.

\bibitem{fus}
Fusco, N. (2015). The quantitative isoperimetric inequality and related topics. \emph{Bulletin of Mathematical Sciences} 5, 517--607.

\bibitem{fumapr} 
Fusco, N., Maggi, F., and Pratelli, A. (2008). The sharp quantitative isoperimetric inequality. \emph{Annals of Mathematics} 168, 941-980.

\bibitem{gar}
Gardner, R.J. (2006). Geometric Tomography, 2nd ed. \emph{Cambridge University Press}.

\bibitem{ind}
Indrei, E. (2024). On the first eigenvalue of the Laplacian for polygons. \emph{J. Mathematical Physics} 65, 041506.

\bibitem{mal} Malagoli, F. (2007). A problem of minimum area for bodies with constrains on width and curvature. \emph{Rendiconti dell'Istituto di Matematica dell'Università di Trieste} 39, 407--420.

\bibitem{oss}
Osserman, R. (1979). Bonnesen-style isoperimetric inequalities. \emph{The American Mathematical Monthly} 86(1), 1--29.

\bibitem{pal}
P\'al, J. (1921). Ein minimumproblem für Ovale. \emph{Mathematische Annalen} 83(3), 311--319.

\bibitem{rolf} Schneider, R (2013). Convex Bodies: The Brunn-Minkowski Theory, 2nd expanded edition. \emph{Cambridge University Press}

\bibitem{tao}
Tao, T. (2001). From rotating needles to stability of waves: emerging connections between combinatorics, analysis, and PDE. \emph{Notices Amer. Math. Soc.} 48(3), 294--303.

\bibitem{yagbol}
Yaglom, I. M., and  Boltyanskii, V. G. (1961). Convex Figures (Trans. PJ Kelly \& LF Walton) Holt. \emph{Rinehart \& Winston, New York, 11}.


\end{thebibliography}
\end{document}